\documentclass[12pt,a4paper]{article}

\usepackage[utf8]{inputenc}
\usepackage[T2A]{fontenc}
\usepackage[english]{babel}
\usepackage{amscd,amssymb,amsthm}
\usepackage{amsfonts,amsmath,array}
\usepackage[margin=2cm]{geometry}
\usepackage{mathtools,mathcomp}
\usepackage{hyperref}

\usepackage{pgf,tikz}
\usetikzlibrary{arrows}

\title{The chromatic number of the plane with an interval \\of forbidden distances is at least 7}
\author{Vsevolod Voronov\footnote{PhD, Researcher, Caucasus Mathematical Center of Adyghe State University, Maikop, Russia; Moscow Institute of Physics and Technology, Dolgoprudniy, Russia}}

\begin{document}

\bigskip

\newtheorem{observ}{Observation}[section]
\newtheorem{lemma}{Lemma}[section]
\newtheorem{prop}{Proposition}[section]
\newtheorem{claim}{Claim}[section]
\newtheorem{defin}{Definition}[section]
\newtheorem{question}{Question}[section]
\newtheorem{theorem}{Theorem}[section]
\newtheorem{remark}{Remark}[section]
\newtheorem{cond}{Condition}[section]
\newtheorem{corr}{Corollary}[section]
\newtheorem{property}{Property}

\maketitle

\abstract{The work is devoted to one of the variations of the Hadwiger--Nelson problem on the chromatic number of the plane. In this formulation  one needs to find for arbitrarily small~$\varepsilon$ the least possible number of colors needed to color a Euclidean plane in such a way that any two points, the distance between which belongs to the interval $[1-\varepsilon, 1+\varepsilon]$, are colored differently. The conjecture proposed by G. Exoo in 2004, states that for arbitrary positive~$\varepsilon$ at least 7 colors are required. Also, with a sufficiently small~$\varepsilon$ the number of colors is exactly 7. The main result of the present paper is that the conjecture is true for the Euclidean plane as well as for any Minkowski plane.
}

\section{Introduction}

This paper considers one of the versions of the classical Hadwiger--Nelson problem on the chromatic number of the plane. The book by A. Soifer \cite{Soifer} is devoted to the history of this problem and its numerous variants. Reviews of recent results for low-dimensional spaces can be found in \cite{Rai1,warsaw}.

Let us consider some set of forbidden distances $\mathcal{D} \subset \mathbb{R}_+$. It is required to find the smallest natural  $k$ for which there exists a partition of the plane into $k$ pairwise non-intersecting subsets, none of which contains a pair of points at the Euclidean distance belonging to $\mathcal{D}$.  In other words, we need to color the plane with $k$ colors such that any two points whose distance between them belongs to $\mathcal{D}$ have a different color, i.e.
\[
\mathbb{R}^2 = M_1 \sqcup M_2 \sqcup \dots \sqcup M_k, \quad \forall i \; \forall x,y \in M_i: \; \|x-y\| \not\in \mathcal{D},
\]
where $\|x-y\|$ is the Euclidean distance between points $x$, $y$, and $M_i$ is the set of points in the plane colored in the $i$-th color.

Let us call the smallest $k$ for which this is possible the \emph{chromatic number of the plane} $\chi_{\mathcal{D}}(\mathbb{R}^2)$ with a forbidden distance set $\mathcal{D}$. In the case $\mathcal{D} = \{1\}$ we have a classical formulation of the Hadwiger--Nelson problem. In 2018, the lower estimate  has been improved~\cite{deGrey,exoo2020chromatic}, and it is now known that

\[
5\leq \chi_{\{1\}}(\mathbb{R}^2) \leq 7.
\]

In recent years, a number of interesting results have been obtained in this area.  In particular, J. Davies \cite{Davies} managed to prove that when $\mathcal{D} = \{1,3,5,7, \dots\}$ it holds that $\chi_{\mathcal{D}}(\mathbb{R}^2) = +\infty$.

The case $\mathcal{D} = [1, b]$ was studied in~\cite{woodall1973distances,townsend2005colouring,currie2015chromatic,warsaw,parts2023}. It has been shown~\cite{townsend2005colouring,currie2015chromatic} that 
\[
 \chi_{[1,b]}(\mathbb{R}^2)\geq 6, \quad \forall b>1.
\]

\begin{figure}
    \centering
    \includegraphics[width=6cm]{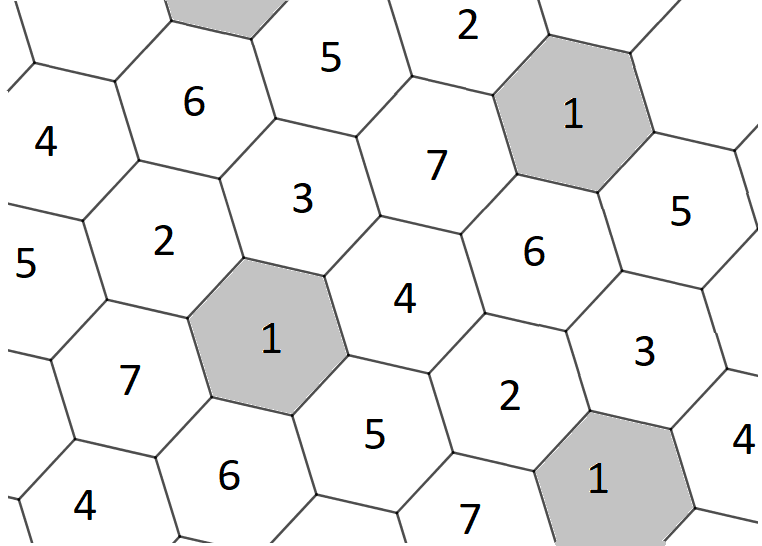}
    \caption{7-coloring of the plane}
    \label{hex7col}
\end{figure}

From the coloring of the tessellation of the plane with regular hexagons (Fig.~\ref{hex7col}), it follows that
\[
 \chi_{[1,b]}(\mathbb{R}^2)\leq 7, \; b \leq \frac{\sqrt{7}}{2}.
\]

Note that with computer calculations, it is possible to find the exact value of the chromatic number  for some ranges, for instance~\cite{parts2023},
\[
 \chi_{[1,b]}(\mathbb{R}^2)=7, \quad 1.085...\leq b \leq \frac{\sqrt{7}}{2} = 1.322...
\]

The hypothesis formulated by G. Exoo in \cite{exoo1}, states that this equality is also true for $b$ arbitrarily close to $1$.  The proof of this hypothesis  is the main result of this paper. Moreover, it is straightforward to generalize  this result to the Minkowski plane, i.e., the metric given by an arbitrary 2-dimensional norm.

Let $\|\cdot \|_U$ be the 2-dimensional norm for which $U = \{x: \; \|x\|_U\leq 1\}$. Clearly, the shape $U$ is convex for any norm. If $U$ is strictly convex, then we will call the corresponding norm strictly convex. 

In the paper by K. Chilakamarri \cite{chilakamarri1991unit} it was shown that in the case of a single forbidden distance and an arbitrary norm  
\[
    4 \leq \chi_{\{1\}}(\mathbb{R}^2; U)\leq 7.
\]

Note that if we are interested in what happens when the ratio of interval bounds tends to~1, it does not matter whether we consider a closed or an open interval. Moreover, when $\varepsilon \to 0$, the cases $\mathcal{D}=[1-\varepsilon,1]$, $\mathcal{D}'=[1-\varepsilon,1+\varepsilon]$, $\mathcal{D}''=[1,1+\varepsilon]$ are equivalent.

\begin{theorem}
\label{thm_main}
For any 2-dimensional norm $\|\cdot\|_U$ and $\varepsilon>0$ it holds that
\[
 \chi_{[1-\varepsilon,1+\varepsilon]}(\mathbb{R}^2; U) \geq 7. 
\]     
\end{theorem}

It is sufficient to show that the lower estimate holds for arbitrarily small positive $\varepsilon$,  or, in other words, that there is no coloring of the plane with 6 colors under these conditions. 

The following statement, which is used in the proof of the main result, may be of particular interest.

\begin{theorem}
\label{thm_bicycle}
    Denote the unit circle centered at $x$ by $T_1(x)=\{y \;:\;\|x-y\|_U=1\}$.   Then for $\varepsilon>0$ and $1<\|u-v\|_U<2$ it holds that
    \[
        \chi_{[1-\varepsilon,1+\varepsilon]}(T_1(u) \cup T_1(v); U) \geq 4.
    \]
\end{theorem}

In the case of a strictly convex norm, we can specify some segment on which the chromatic number is equal to 7. In particular, for the Euclidean norm, at least the following can be asserted.

\begin{corr}
Let $0<\varepsilon \leq   \frac{\sqrt{7}-2}{\sqrt{7}+2} = 0.138...$. Then
\[
 \chi_{[1-\varepsilon,1+\varepsilon]}(\mathbb{R}^2) = 7. 
\]
\end{corr}

In the reasoning that we will use further, it is sufficient to consider a bounded region of the plane. Therefore, let us formulate in terms of the distance realization  \cite{hadwiger2015combinatorial,woodall1973distances} the result which can be obtained by the same tools without additional efforts. We say that the set $M\subset\mathbb{R}^2$ realizes a distance $d$ if there are points $x,y \in M$ for which $\|x-y\|_U=d$.
 \begin{theorem}
 \label{thm_cover}
      Suppose that the disc of radius 3 defined by an arbitrary 2-dimensional norm is covered by 6 closed sets. Then at least one of these sets realizes the unit distance.
 \end{theorem}

In the considered formulation of the problem we can assume, as will be explained later, that the plane is divided into some regions by Jordan curves or even polylines, and the interior of each region is monochromatic. In other words, we consider a coloring of a map on the plane. It seems that the problem of the chromatic number of a map in the plane with a single forbidden distance \cite{townsend2005colouring,woodall1973distances} is harder than the problem with a forbidden interval. Given a single forbidden distance, such a coloring in $k$ colors may exist even if there is no coloring in $k$ colors with a short forbidden interval.

Another well-known example of conditions imposed on the coloring of the plane with a single forbidden distance is the requirement that the sets of points painted in each color in some bounded region are measurable. The problem on measurable chromatic number seems to be substantially more difficult than the one under consideration. 

Also note that C.~Thomassen proved~\cite{thomassen1999nelson} that  $7$ colors are required  for a map-type coloring of the two-dimensional surface with some metric, if the diameter of the surface is large enough,  all the regions are simply connected, the diameter of each region is less than 1, and any two regions of the same color are at a distance greater than 1 (the definition of a ``nice'' coloring in the mentioned article). The latter assumption is not used in the present paper. 

If we denote, respectively, $\chi_{mes}(\mathbb{R}^2)$, $\chi_{map}(\mathbb{R}^2)$, $\chi_{[1-\varepsilon,1+\varepsilon]}(\mathbb{R}^2)$, $\chi_{nice}(\mathbb{R}^2)$ the measurable chromatic number, the chromatic number for the coloring of the map type, the chromatic number for a short forbidden interval, then, taking into account the results of the present paper, the following estimates are currently known:
\[
      5 \leq \chi_{\{1\}}(\mathbb{R}^2) \leq \chi_{mes}(\mathbb{R}^2) \leq  \chi_{map}(\mathbb{R}^2) \leq 7 \leq  \chi_{[1-\varepsilon,1+\varepsilon]}(\mathbb{R}^2), 
\]      
\[
   \quad 6 \leq \chi_{map}(\mathbb{R}^2)\leq \chi_{nice}(\mathbb{R}^2)=7. 
\]

\section{Preliminaries}

Recall that a norm in $\|\cdot\|$  on the vector space $\mathbb{R}^n$ is a real-valued function satisfying the following axioms:

\begin{enumerate}
    \item $\|x\| \geq 0$ for all $x \in \mathbb{R}^n$, and $\|x\|=0$ iff $x=0$.
    \item $\|\alpha x\| = |\alpha| \cdot \|x\|$ for all $x \in \mathbb{R}^n, \;  \alpha \in \mathbb{R}$.
    \item For all $x, y \in \mathbb{R}^n$ the triangle inequality holds: $\|x+y\| \leq \|x\|+ \|y\|$.
\end{enumerate}

Any norm is uniquely determined by the unit ball $U = \{x: \; \|x\|\leq 1\}$. The unit ball is a closed, bounded, symmetric, convex set with non-empty interior. The norm $\|\cdot\|_U$ is called \emph{strictly convex}, if the equality $\|x+y\|= \|x\|+\|y\|$ implies that $x$ and $y$ are linearly dependent.  Equivalently, strict convexity holds if the boundary of $U$ does not contain a line segment.

In the following, unless otherwise specified, we assume that $\|\cdot\|$ is a strictly convex 2-dimensional norm. For brewity, sometimes we use this notation  instead of $\|\cdot\|_U$. Throughout the article, a circle refers to a \emph{circle defined by the given norm}, i.e. the unit circle centered at the point $x\in \mathbb{R}^2$~is
\[
    T_1(x)= \{y\in \mathbb{R}^2  \;:\; \|y-x\|_U = 1 \}.
\]

Similarly, the disc of radius $r>0$ centered at $x$ is
\[
    U_r(x)= \{y\in \mathbb{R}^2  \;:\; \|y-x\|_U \leq r \}.
\]

Next we need some properties of a strictly convex 2-dimensional norm. 

\begin{property}[\cite{martini2001geometry}, Proposition 14, p. 106.]
    If $x \neq y$, $\|x-y\|<2$ then  two unit circles $T_1(x)$, $T_1(y)$ intersect at exactly two points.
    \label{property1}
\end{property}

\begin{property}[\cite{martini2001geometry}, Proposition 31 (Monotonicity Lemma), p. 115.]
    Let $u \neq 0$, $v,w \in T_1(x)$, $u \neq w$ and the ray $xv$ lies in the angle $uxv$. Then $\|u-v\| \leq \|u-w\|$, and the equality is possible iff $v=w$.
    \label{property2}
\end{property}

In most cases, we will use the following

\begin{corr}[\cite{martini2001geometry}, p. 114.]
    For a fixed point $x \in T_1(0)$ and a variable point $y \in T_1(0)$, the distance $\|x-y\|$ continuously increases as $y$ ranges from $x$ to the antipode $-x$.
    \label{property2_cor1}
\end{corr}

\begin{corr}
\label{cor_arc_inclusion}
    If two arcs $pq$, $p'q'$ are less than half a circle, and $p'q'$is a proper subset of $pq$, then $\|p'-q'\|<\|p-q\|$. 
\end{corr}

To show this, we can move $p$ to $p'$, then $q$ to $q'$ along the circle, and the distance between the ends of the arc will continuously decrease according to Property \ref{property2}. One could also prove this inequality using Proposition 32 in \cite{martini2001geometry} (p. 116). 

\begin{property}[\cite{martini2001geometry}, Propositions 33 and 34, p. 116]
    For any point $x\in T_1(0)$ the equilateral hexagon having a vertex $x$ can be inscribed in $T_1(0)$.
    \label{property3}
\end{property}



\begin{observ}
    \label{observ_chord}
    Let $x_0 \in T_1(0)$, $x(t)=t x_0$, $0<t<2$. Denote the intersection points of $T_1(0)$ and $T_1(x(t))$ by $u_1(t), u_2(t)$. Then the length of the common chord $f(t)=\||u_1(t)-u_2(t)\|$ is monotonically decreasing when $t$ ranges from $0$ to $2$. Furthermore, if $t<1$, then $f(t)>1$.
\end{observ}

\begin{proof}
Indeed, $f(t) \to 0$ at $t \to 2$, and $f(t)$ is positive by Property \ref{property2} at $0<t<2$. Furthermore, $u_1(t)$  cannot take the same value at two different points $t=t_1$, $t =t_2$, otherwise we have three collinear points $0$, $x(t_1)$, $x(t_2)$ on the circle $T_1(u_1(t_1))$, which is impossible. The same is true for $u_2(t)$. Thus, one of the points $u_1, u_2$ moves clockwise on the circle, the other counterclockwise, when $t$ increases.  

In the above construction, $u_1(t_1)u_2(t_1) \subset u_1(t_2)u_2(t_2)$ at $t_1>t_2$, which means that $f(t)$ is monotonically decreasing (by Corollary \ref{cor_arc_inclusion}).

Finally, note that if $t=1$, then $u_1(1)$ and $u_2(1)$ are two vertices of the hexagon from Property \ref{property3} taken through one. This arc contains an arc whose distance between the ends is 1 (namely, side of the hexagon). Thus, $f(1)>1$ and $t<1$ implies $f(t)>1$.
\end{proof}

Note that for the reasoning used in the proof of the main result, all the strictly convex norms are the same. To consider the case of a non-strictly convex norm, it is enough to approximate it by a strictly convex norm and reduce the length of the forbidden interval. The following statement has been proved in a number of sources in a stronger form than is required, and we will not focus on it.

\begin{lemma}[follows from \cite{ahmadi2019polynomial}, Theorem 3.1]
    Let $\|\cdot\|_H$ be an arbitrary norm. Then for arbitrary $\delta>0$ there is a strictly convex norm $\|\cdot\|_U$ such that
    \[
         (1-\delta) \|x\|_U \leq \|x\|_H \leq \|x\|_U, \; \forall x \in \mathbb{R}^2 .
    \]
    \label{lemma_strictly_convex}
\end{lemma}

Next we proceed to consider the colorings of the plane and and introduce conditions imposed on the colorings by using a strictly convex norm $\|\cdot\|$.

\begin{defin}
Let us call a coloring of the plane $\varphi: \mathbb{R}^2 \to \{1, 2, \dots, k\}$ for some forbidden interval $\mathcal{D}=[1-\varepsilon,1+\varepsilon]$ \textbf{proper}, if there is no pair of points $u,v \in \mathbb{R}^2$ such that $\|u-v\| \in [1-\varepsilon,1+\varepsilon]$ and $\varphi(u)=\varphi(v)$.      
\end{defin}

If this condition is satisfied for some set $A \subseteq \mathbb{R}^2$, we will also say that the set A is properly $k$-colored. If only colors from $\{c_1, c_2, \dots, c_m\}\subset \{1,2, \dots, k\}$ occur in $A$, then we say that $A$ is properly colored in $c_1, c_2, \dots, c_m$.

\begin{prop}
Suppose that for $\mathcal{D}=[1-\varepsilon,1+\varepsilon]$ there exists some proper $k$-coloring of the plane. Let the plane, in addition, be divided into tiles of arbitrary shape, different tiles can be of different shapes,  the diameter of each tile does not exceed $h$, and $2h<\varepsilon$. Then for the interval of positive length $\mathcal{D}_h=[1-(\varepsilon-2h),1+(\varepsilon-2h)]$ there exists a proper $k$-coloring of the plane, in which the interior of each tile is monochromatic, and the border points are colored in some color of the adjacent tiles. 
    \label{prop_discrete}
\end{prop}

\begin{proof}In fact, it is sufficient to color the interior of each tile in any of the colors that occur in it. If the original coloring did not have a pair of same colored points at a distance from $[1-\varepsilon,1+\varepsilon]$, then the new coloring will not have a pair of same colored points at a distance belonging to $\mathcal{D}_h$.
\end{proof}

\begin{defin}
    We will call a (hexagonal) \textbf{discrete coloring} with parameter $h>0$ the coloring of the plane defined by regular hexagons of diameter $h$, in which the interior of each hexagon is monochromatic, and the common points are arbitrarily colored in the color of any of the hexagons to which they belong. 
\end{defin}

Note that if the statement of Theorem~\ref{thm_main} is proved for discrete colorings of the mentioned kind for arbitrarily small $h$, then it will also be proved in the general case. 



\begin{observ}
Consider a proper discrete coloring of the plane. If some connected monochromatic region $Q$ of color $c^*$ is not 1-connected (i.e. it contains inclusions of a different color, and it is not homeomorphic to a disk), it is possible to recolor everything contained inside the outer boundary of $Q$ in color $c^*$. After that the new coloring remains proper, and the region $Q' \supset Q$ of color $c^*$ will be 1-connected. 

\end{observ}

\begin{proof}Suppose that the coloring is no longer proper after recoloring. This means that there are two points $x, y$ of the same color at the distance $d\in[1-\varepsilon,1+\varepsilon]$. If $x, y \in Q'$, then the diameter $\operatorname{diam} Q=\operatorname{diam} Q'>d$, which is impossible. If one of the points does not belong to $Q'$, then other point lies in $Q'$ (for certainty, let $x \not\in Q'$, $y \in Q'$). The circle of radius $d$ centered at $x$ intersects the boundary of $Q'$, which coincides with the boundary of $Q$. Thus, there is a point $z \in Q$, $\|z-x\|=d$, $c(z)=c(x)$. Hence the original coloring is not proper, a contradiction.
\end{proof}

By repeating these steps, we can make all regions 1-connected in any bounded subset of the plane. Consequently, we can assume that this is also true for the whole plane.

\begin{observ}
    If there exists a proper discrete coloring, then there exists also such a proper discrete coloring that all maximal connected monochromatic regions are 1-connected.
    \label{obs_1connected}
\end{observ}

\begin{defin}
A point $x$ has \textbf{multicolor} $C(x)\subseteq\{1,2, \dots, k\}$ in some coloring of the plane $\varphi$ if there are points of each color belonging to $C(x)$ in an arbitrarily small neighborhood of $x$, and $C(x)$ is the maximal set by inclusion for which this condition is satisfied.    
\end{defin}

\begin{defin}
    The number $|C(x)|$, i.e. the number of colors that occur in an arbitrarily small neighborhood of $x$, will be called the \textbf{chromaticity} of point $x$. Points of chromaticity 3 will also be called trichromatic, and points of chromaticity 2 will be called bichromatic.
\end{defin}

Because of the mentioned above, we can assume that in any coloring of interest, the set of trichromatic points is discrete, and the set of bichromatic points is a union of segments. 

The following lemma has been proved in various ways in \cite{brown1994colorings,currie2015chromatic,warsaw}. A short proof that does not depend on the choice of norm can be derived from Proposition \ref{prop_discrete} and Observation~\ref{obs_1connected}. 

\begin{lemma}
\label{lemma_trichromatic_pt}
A proper coloring of the plane in an arbitrary number of colors with a forbidden interval of positive length contains a point of chromaticity of at least 3. We can find such a point in a circle of unit radius centered at a given point $x$.
\end{lemma}

\begin{proof}
Indeed, if we proceed to the discrete version of coloring with 1-connected sets and a parameter $h>0$, then any point $x$ is contained in a monochromatic set bordered by sets of two other colors. Observe that the boundary of a monochromatic region lies entirely inside $T_1(x)$. Hence, in the discrete coloring the unit disc centered at $x$ contains a trichromatic point. Considering the sequence of discrete colorings as $h \to 0$, we see that this circle contains some limit point, which will be (at least) a trichromatic point in the original coloring.
\end{proof}    

A similar statement was proved for the arbitrary dimension in the case of Euclidean norm~\cite{slice3}.

\begin{lemma}
A proper coloring of $\mathbb{R}^n$ with $\mathcal{D}=\{1\}$ contains a point of chromaticity at least $n+1$, and there is at least one such point in every regular $(n+1)$-simplex with the side length $a = \sqrt{2n(n+1)}$.
\end{lemma}

In further reasoning, the mutual arrangement of the trichromatic points will play an important role. 

\begin{prop}
    Let $\varphi$ be a proper coloring of the plane in 6 colors with a forbidden interval $[1-\varepsilon,1+\varepsilon]$. Then any two trichromatic points $u, v$ for which $C(u) \cap C(v) = \emptyset$ are at a distance $\|u-v\| \geq 2+2\varepsilon$.
    \label{prop3col1}
\end{prop}

\begin{proof}
    Otherwise, there will be a point $w$ for which \[\|u-w\| = \| v-w \| = l \in (1-\varepsilon,1+\varepsilon).\] 
Then $w$ cannot be colored because \[C(u) \cup C(v) = \{1, 2, \dots, 6\}. \qedhere\]
\end{proof}


\section{Colorings of the circle}

Denote by $T_1(x)$ a circle of unit radius centered at point $x$. If in a proper $6$-coloring of the plane the point $x$ is trichromatic, for example $C(x)=\{4,5,6\}$, then $T_1(x)$ must be colored in 1, 2, 3.  In fact, we might not have to prove that the proper 3-coloring of the unit circle exist, since otherwise the main result is immediately obtained. But without it the further presentation would look a little strange, although formally correct.

\begin{figure}
    \centering
    \includegraphics[width=8cm]{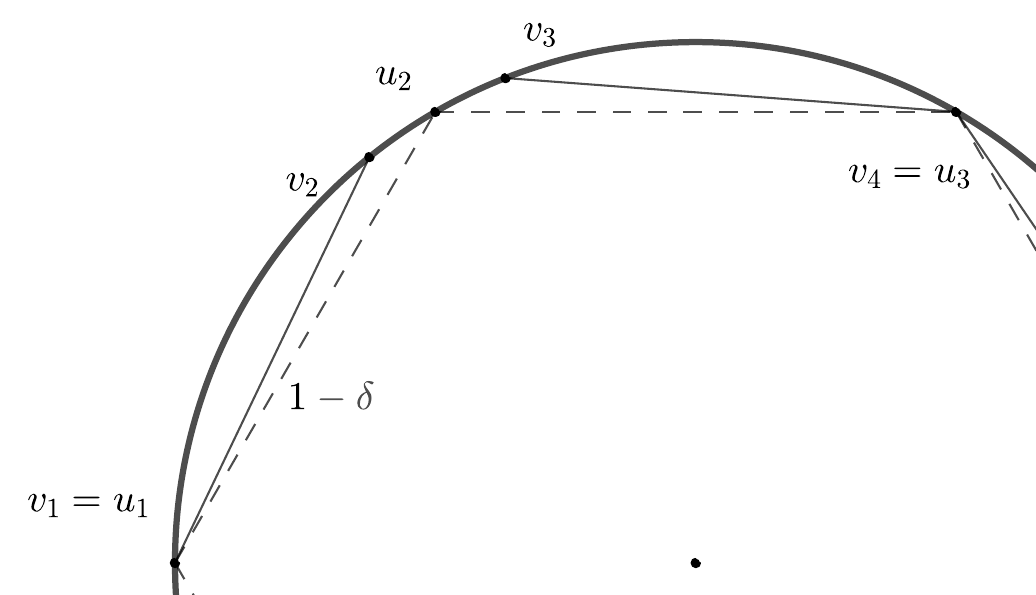}
    \caption{Construction of 9-gon in Proposition \ref{prop_9gon}}.
    \label{fig_9gon}
\end{figure}

\begin{prop}
    \label{prop_9gon}
    It is possible to properly color $T_1(x)$ in 3 colors, if $\varepsilon$ is small enough.    
\end{prop}

\begin{proof}
    Consider the equilateral hexagon $u_1, \dots u_6$ with unit side length inscribed in $T_1(x)$ (see Property \ref{property3}), and construct an almost degenerate 9-gon with three short sides, which is also inscribed in $T_1(x)$. Let $v_1=u_1$, $v_4=u_3$, $v_7=u_5$. Next to each of the vertices $u_2$, $u_4$, $u_6$ we put two new ``dependent'' vertices in such a way that $v_1 v_2=v_3 v_4=v_4 v_5 = v_6 v_7 = v_7 v_8 = v_9 v_1 = 1 - \delta$ (Fig. \ref{fig_9gon}). For convenience, define $v_{10}=v_1$, $v_{11}=v_2$. Let the color of the arc $v_{i} v_{i+1}$ be $c_i = (i\mod 3) + 1$, $i=1,2, \dots, 9$.

    By Property \ref{property2} and Corollary \ref{cor_arc_inclusion}, the arc lengths $f_i(\delta)=\|v_i -v_{i+1}\|$ and the distances between the arcs of the same color $g_i(\delta)= \|v_i -v_{i+2}\|$ are monotone in the neighborhood of zero. Moreover, $0<f_i(\delta)<1$, $g_i(\delta)>1$ for all $\delta>0$, $\delta<\delta_1$. Choose an arbitrary $\delta$ belonging to the interval $(\delta, \delta_1)$. Then there exists $\varepsilon>0$ such that $0<f_i(\delta)<1-\varepsilon$, $g_i(\delta)>1+\varepsilon$, $i=1,2, \dots, 9$. Thus, the coloring of $T_1(x)$ defined above will be proper for the forbidden interval of distances $[1-\varepsilon, 1+\varepsilon]$.
\end{proof}

Here we are interested in some properties of the proper 3-colorings of the unit circle. Note that by Proposition \ref{prop3col1} there are no trichromatic points on $T_1(x)$. As before, to avoid some technical difficulties, we will consider only the colorings of the circle generated by a proper discrete coloring of the plane, i.e., a coloring with sufficiently small monochromatic hexagonal tiles. For such colorings let us formulate the following

\begin{observ}
   A proper 3-coloring of the circle $T_1(x)$, generated by a hexagonal discrete coloring of the plane, contains a finite number of bichromatic points.
\end{observ}

This follows from the fact that the borders of the tiles are line segments. The unit circle $T_1(x)$ cannot contain a line segment in the case of a strictly convex norm.

\begin{observ}
If the diameter of the arc $pq \subset T_1(x)$ is at least $1$, and $pq$ contains a point of color $1$, then $pq$ also contains  a bichromatic point $u$, such that $1 \in C(u)$.
\end{observ}

This follows directly from the definition of multicolor and from the fact that such an arc cannot be monochromatic.

\begin{observ}
    One can assume that the circle does not contain isolated points of any color, i.e. points of color $i$, such that in some neighborhood there are no other points of color $i$. 
\end{observ}

In fact, such points can be recolored to the color of their neighborhood preserving the proper coloring of the circle.
 
\begin{defin}
We call the alternation of bichromatic points (and arc colors) \textbf{cyclic} if one does not encounter two arcs of the same color with exactly one arc between them  when going around a circle. I.e. the colors of the arcs when going around the circle are either 1, 2, 3, 1, 2, 3, ... or 1, 3, 2, 1, 3, 2, ..., and there are no subsequences of the form $a, b, a$. We will call a coloring with cyclic alternation of colors a \textbf{cyclic coloring}.
\end{defin}

Here we define a transform which, for an arbitrary strictly convex norm, can be regarded as some analogue of the rotation of the Euclidean circle.

\begin{defin}
    \label{def_rotation}
    Let us call a counterclockwise $d$-rotation $\psi_d(\cdot)$, $0<d<2$ such a mapping of the circle $T_1(x)$ into itself that  $\forall u \in T_1(u): \;\|u - \psi(u)\|=d$ is satisfied, and of the two possible points, the point closest to $u$ is chosen when traversing the circle counterclockwise. Similarly, if the point closest to $u$ is chosen when traversing clockwise, then we call such a mapping $\psi_{-d}(\cdot)$ a clockwise $d$-rotation. The most important case is 1-rotation.
\end{defin}

Correctness of Definition \ref{def_rotation} is provided by Property \ref{property2}. Moreover, from this Property~\ref{property2} one can derive that the mappings $\psi_d$, $\psi_{-d}$ are one-to-one. Note that \[\forall u \in T_d(x): \; \psi_{-d}(\psi_d(u))=u,\] thus $\psi_{-d} = \psi_d^{-1}$. Also note that and it follows from Property \ref{property3} that $\psi_1^6= \psi_1^0=\operatorname{Id}$.

\begin{defin}
    Let $u_1v_1$, $u_2 v_2$ are arcs of the unit circle such that $\|u_1-v_1\|<1$, $\|u_2-v_2\|<1$. We will say that arc $u_2 v_2$ is shorter than arc $u_1 v_1$ and use the notation $u_2 v_2 < u_1 v_1$ if there is such $k\in \{0,1, \dots, 5\}$ that $u_2 v_2$ is a proper subset of arc $u_1^{(k)}v_1^{(k)}$, where $u_1^{(k)}=\psi_1^k(u_1)$, $v_1^{(k)}=\psi_1^k(v_1)$.
\end{defin}

If the distance between the ends of the arc $uv$ is less than 1, its images $u^{(k)}v^{(k)}$, $k=0,1, \dots, 5$ are paiwise disjoint. Hence, the arc $uv$ cannot be shorter than itself. If $u_2 v_2 < u_1 v_1$, then for the union of images we have
 \[
     \bigcup_{k=0}^5 u_2^{(k)}v_2^{(k)} \subsetneq \bigcup_{k=0}^5 u_1^{(k)}v_1^{(k)}.
 \]    
If $u_2 v_2 > u_1 v_1$, then the inclusion must be reversed. Also from the properties of set inclusion we can deduce that $u_1v_1>u_2 v_2$, $u_2 v_2> u_3 v_3$ implies $u_1 v_1 > u_3 v_3$. Thus, this relation has the properties of irreflexivity, antisymmetry, transitivity, and we obtain a strict partial order defined on the arcs of the unit circle of diameter less than 1. 

 
\begin{prop}
An arbitrary proper coloring of the circle $T_1(x)$ with a finite number of bichromatic points can be reduced to a proper cyclic coloring after a finite number of recoloring of arcs between bichromatic points.  \label{circ_recolor}  
\end{prop}

\begin{proof}
If the alternation is not cyclic, we will find a triple of arcs with colors 1, 2, 1 (without loss of generality). Denote by $u$, $v$  the bichromatic points at the ends of the arc of color $2$. Let $u'=\psi_{-1}(u)$, $v'=\psi_{-1}(v)$ be points obtained by going clockwise by~$1$. Similarly, we determine $u''=\psi_{1}(u), v''=\psi_{1}(u)$ by going counterclockwise. Then $u'$, $v'$, $u''$, $v''$ together with some neighborhood of these points have color 3. We have to deal with two cases: either it is possible to recolor the arc $uv$ with color 1 while keeping the coloring proper, or at least one of the arcs $u'v'$, $u''v''$, contains arcs that alternate non-cyclically. 
Among them, there is an arc $u_1 v_1$ that separates two arcs of the same color, and $u_1 v_1 < uv$. In this case, only an arc $u_2 v_2$ that is shorter than $u_1 v_1$ can prevent $u_1 v_1$ from being recolored. Since there are a finite number of arcs, we can choose the one that is shorter than all the others to which we can compare it using partial order.

Then nothing prevents its recoloring in the color of neighboring arcs. In a finite number of steps all such arcs will be eliminated, and we will have a cyclic alternation.
\end{proof}

\begin{defin}
    Denote by $M_{ab}$ the set of points of multicolor $\{a,b\}$ on the circle $T_1(x)$. 
\end{defin}   

\begin{prop}
    When $T_1(x)$ is properly colored, each of the sets $M_{12}$, $M_{23}$, $M_{13}$ contains at least three points. Moreover, one can choose three points of each of the multicolors in such a way that the pairwise distances  will be greater than 1.    
    \label{multicol}
\end{prop}


\begin{proof}
    Suppose the contrary. Using Proposition~\ref{circ_recolor}, we make the color alternation cyclic, and in this process some bichromatic points will disappear.
    Then the number of bichromatic points of each multicolor will be the same. But if this number is less than 3, then the properly colored circle is divided into 3 or 6 monochromatic arcs, which is impossible, because in this case one of the arcs has diameter at least $1$ and contains an edge of the graph (see Property \ref{property3}).

Suppose that the alternation is cyclic and there is no triplet of points $u,v,w \in M_{12}$ for which  the pairwise distances between $u,v,w$ are greater than $1$. 

Then we find the points $u_1, u_2 \in M_{12}$, $\|u_1-u_2\|<1$. Since the alternation of the bichromatic points is cyclic, one can find the points $v \in M_{23}$, $w \in M_{31}$  on the arc $u_1 u_2$. Consider the points $v'=\psi_1(v), w'=\psi_1(w)$. Observe that $c(v')=1$, $c(w')=2$, and in a cyclic coloring the arc of color 1 is followed by the arc of color 2. Hence the arc $v'w'$ contains the point $u_3 \in M_{12}$, and $\|u_2 - u_3\|<1$. By induction, the distance between any two neighboring points in $M_{12}$ when going around the circle is less than 1. Then  divide the circle into 6 arcs by the vertices of a equilateral hexagon (by Property~\ref{property3}). Here we assume that the hexagon was constructed in such a way that the vertices do not coincide with points from $M_{12}$. After that we are able to choose three points from $M_{12}$, belonging three distinct arcs, the pairwise distances between which are equal to 1.
\end{proof}

\begin{prop}
Let the unit disc $U_1(x)$ with the boundary $T_1(x)$ be properly 3-colored. Some pairs of points from $M_{12}$, are connected by pairwise non-intersecting  polylines of multicolor $\{1,2\}$, belonging to the interior of the circle (i.e. the boundaries between regions of colors 1 and 2). Then there are three points $u,v,w \in M_{12}$ that are at pairwise distances of at least 1 and are not connected to other points from $M_{12}$ by polylines.
    \label{prop_connect}
\end{prop}

\begin{proof}
Under these conditions, a pair of points $p,q\in M_{12}$ connected by a polyline of multicolor $\{1,2\}$ is at a distance of less than 1. Consider a region $Q \subset U_1(x)$ of diameter less than 1, bounded by a polyline $pq$ and an arc of a circle. 

Then any circle of unit radius with center $z \in U_1(x)$ that intersects $Q$ also intersects the polyline $pq$. Indeed, the circle $T_1(z)$ must intersect the boundary of $Q$. Note that $T_1(z)$ cannot intersect an arc of diameter less than 1 by two points by Observation~\ref{observ_chord}. If $T_1(z)$ intersects the arc $pq$ by one point, then one of the points p, q lies inside the circle and the other outside, and hence the circle intersects a polyline. If $T_1(z)$ does not intersect the arc $pq$, we immediately have the required.

Thus, the whole region $Q$ can be colored in 1 or 2 preserving the proper coloring of \emph{the unit disc} $U_1(x)$. Let us recolor the $Q$ in the color that is adjacent to the polyline $pq$ from the outside of $Q$. The points $p, q$ will cease to be bichromatic. Repeating this operation as long as possible, we obtain a proper coloring of the circle without pairs of connected points. In this coloring, according to Proposition~\ref{multicol}, we will find the required triple of points.
\end{proof}

\section{Coloring of complementary arcs}

In this section, we will be interested in the proper 3-coloring of the two arcs along which the ends of the unit segment are running. 

\begin{defin}
    
Let the functions $\gamma_1:[0,1] \to \mathbb{R}^2$, $\gamma_2:[0,1] \to \mathbb{R}^2$ be continuous, and besides,
\[
    \|\gamma_1(t)-\gamma_2(t)\| \in (1-\varepsilon, 1+\varepsilon), \quad \forall t \in [0,1].
\]

\begin{figure}
    \centering
    \includegraphics[width=12cm]{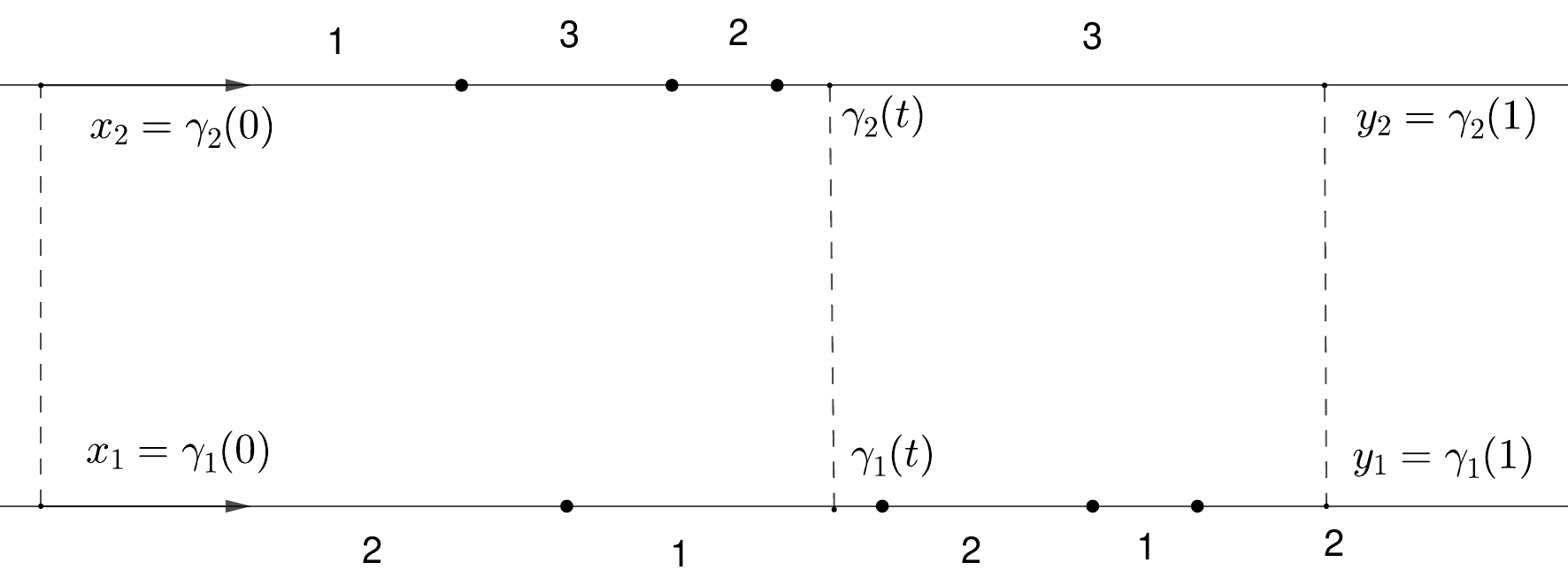}
    \caption{A pair of complementary arcs colored in $1, 2, 3$}
    \label{fig_arcs}
\end{figure}

In this case, the directed arcs $x_i y_i = \gamma_i([0,1]), i=1,2$ will be called \textbf{complementary}. We will also say that the directed arc $x_1 y_1$ is complementary to the directed arc $x_2 y_2$, and vice versa.
\end{defin}

Furthermore, suppose that the arcs $x_1 y_1$, $x_2 y_2$ are properly colored in $1, 2, 3$ with the forbidden interval $[1-\varepsilon, 1+\varepsilon]$, and that the number of bichromatic points on these arcs is finite (Fig.~\ref{fig_arcs}). It follows, in particular, that there are no trichromatic points. Let the coloring, as above, be defined by the function $\varphi$ with values in the set $\{1,2,3\}$. 

Note that here we consider only the coloring of the arcs, not the whole plane. Accordingly, we determine the chromaticity of the points on the arcs only by the colors of the points belonging to the arcs. 

\begin{defin}
    Let the point $x = \gamma(t_0)$, $0<t_0<1$ be bichromatic, and 
    \[
        \lim_{t\to t_0-0} \varphi(\gamma(t)) = a, \quad \lim_{t\to t_0+0} \varphi(\gamma(t)) = b.
    \]
We will call the index of the point $x$ the number 
    \[
    \operatorname{ind} x = \begin{cases}
            1, & \text{if } b \equiv a+1 \pmod{3}, \\\\
            -1, & \text{otherwise}.            
    \end{cases}    
    \]
\end{defin}

\begin{defin}
    Let the \textbf{index} $\operatorname{Ind} x y$ of a directed arc $x y$ be the sum of the indices of all bichromatic points belonging to $x y=\gamma([0,1])$.
\end{defin}

It follows directly from the definition that
\begin{observ}
    If $x_i y_i = \gamma_i([0,1]), i=1,2$ and $\gamma_2(t)=\gamma_1(1-t)$, then 
    \[
        \operatorname{Ind} x_2 y_2 = -\operatorname{Ind} x_1 y_1.
    \]
\end{observ}

Let us prove a less trivial statement.
\begin{prop}
    The indices of the complementary arcs $ x_1 y_1=\gamma_1([0,1])$, $ x_2 y_2=\gamma_2([0,1])$ differ by no more than 1. Moreover, if we define the color of the starting points
\[
a_i = \varphi(\gamma_i(0)), \quad i=1,2,
\]
    then we have the following implications
\[
       a_2 \equiv a_1+1 \pmod{3}  \Rightarrow \operatorname{Ind} x_1 y_1 - \operatorname{Ind} x_2 y_2 \in \{-1, 0\};
\]
\[
       a_2 \equiv a_1-1 \pmod{3}  \Rightarrow \operatorname{Ind} x_1 y_1 - \operatorname{Ind} x_2 y_2 \in \{0,1\}.
\]
\label{prop_ind}
\end{prop}

\begin{figure}[ht!]
    \centering
    \includegraphics[width=6cm]{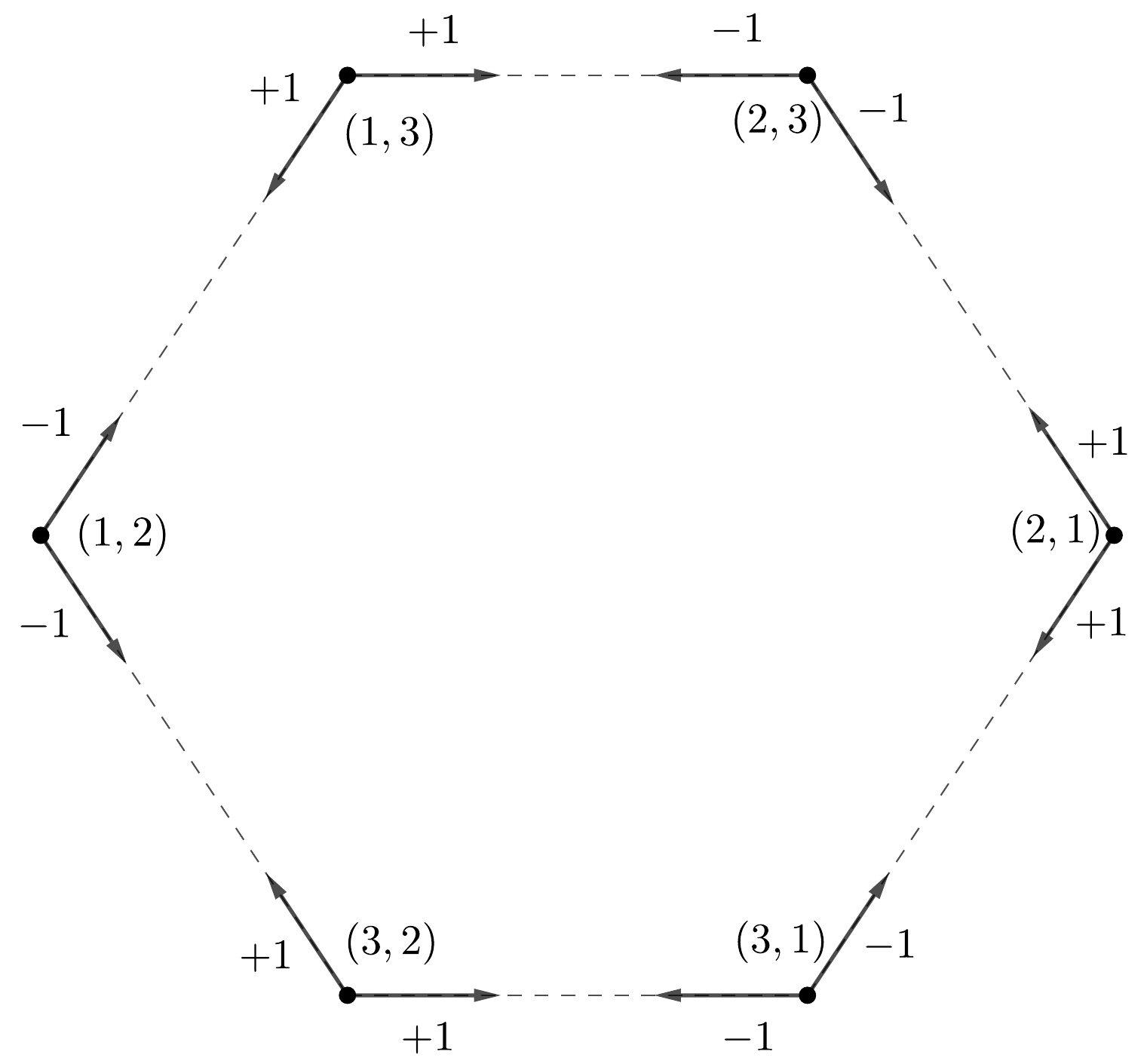}
    \caption{Admissible transitions between pairs of colors}
    \label{graphc6}
\end{figure}

\begin{proof}
It is enough to apply induction by the number of bichromatic points, arranging them by increasing $t$. Let us follow a pair of colors $$(\varphi(\gamma_1(t)),\varphi(\gamma_2(t)).$$ Then the sequence of bichromatic points at $t$ varying from $0$ to $1$ corresponds to some path on the 6-vertex graph $C_6$, shown in Fig.~\ref{graphc6}. It is easy to see that when traversing the graph, the accumulated index difference
    \[
        \delta(t) = \operatorname{Ind}\gamma_1([0,t]) - \operatorname{Ind}\gamma_2([0,t])
    \]
    increases by $1$ when we leave vertex $(1,2)$, $(2,3)$ or $(3,1)$, and decreases by $1$ otherwise.
    Then on any path, the signs of the $+1$, $-1$ summands will alternate. If we start with any of the pairs $(1,2)$, $(2,3)$, $(3,1)$, then the first summand will be negative, otherwise it will be positive. 
\end{proof}





\section{Chromatic number of the bicycle}
\label{sect_bike}
We will call a bicycle a union of two circles. Let us prove the statement of the Theorem~\ref{thm_bicycle} on the chromatic number of the bicycle, which plays a key role in the whole proof. Namely,
    \[
        \chi_{[1-\varepsilon,1+\varepsilon]}(T_1(u) \cup T_1(v)) \geq 4, \quad 1<\|u-v\|<2.
    \]
    \label{bicycle}

\begin{figure}[ht!]
    \centering
    \includegraphics[width=9cm]{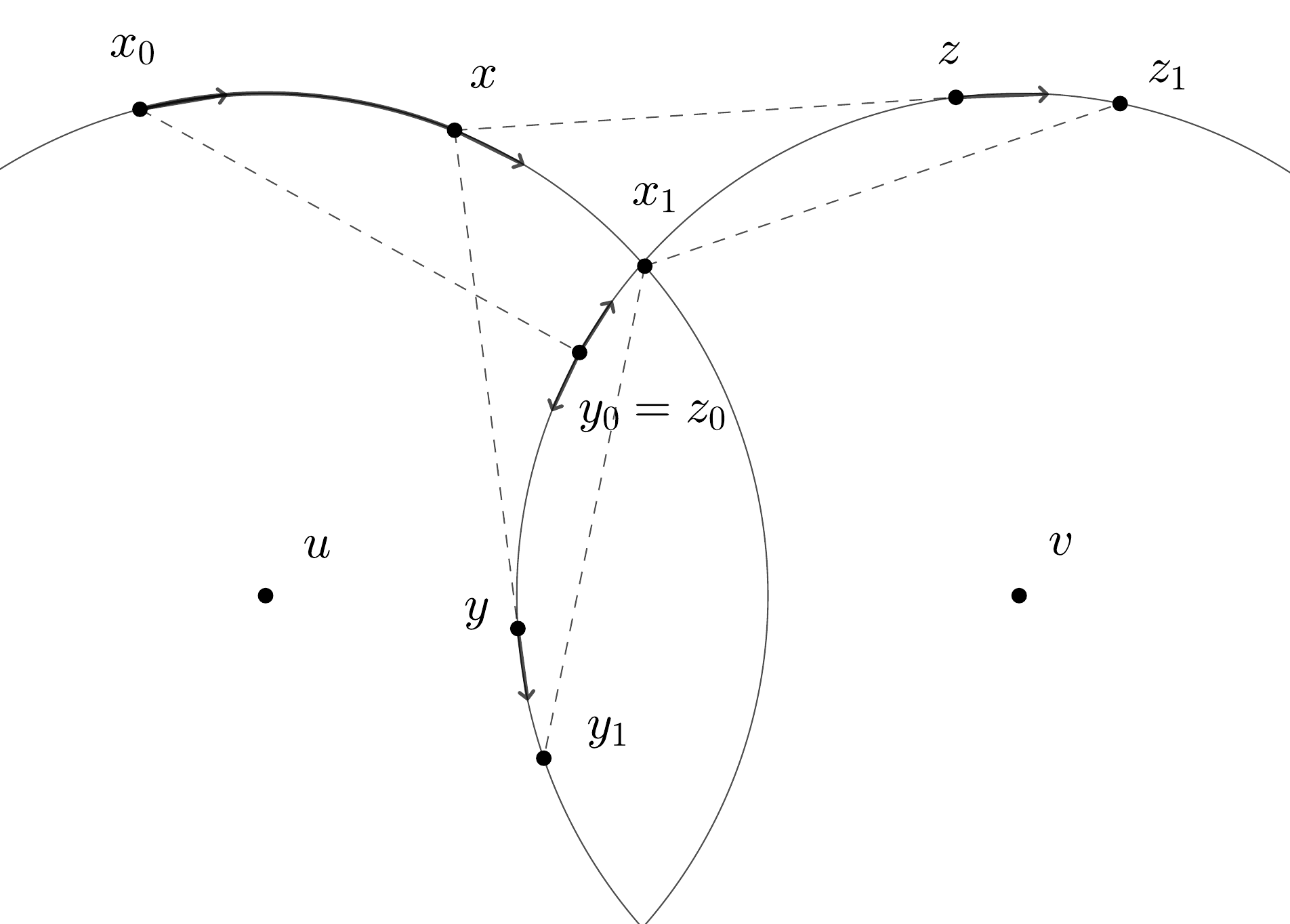}
    \caption{Complementary arcs on the bicycle}
    \label{bicycle_arcs}
\end{figure}

\begin{proof}
    We are going to prove that 3 colors are not enough. Suppose the contrary, let the bicycle be properly colored in 3 colors. Let $x \in T_1(u)$, $\|x-v\|<2$ and find two different points $y,z \in T_1(v)$, for which  
    \[\|x-y\|=\|x-z\|=1.\] 
    
    Observe that under these assumptions we have $|T_1(x) \cap T_1(v)|=2$ by Property \ref{property1}, and the intersection depends continuously on $x$ by Property \ref{property2}. Consider the motion $x=x(t)$  along the circle $T_1(u)$ from the position $x(0)=x_0$, at which $y(0)=y_0 = z(0)=z_0$, to the point $x(1)=x_1 \in T_1(u) \cap T_1(v)$ (Fig.~\ref{bicycle_arcs}). Let
    \[
      \|x(t)-y(t)\| = \|x(t)-z(t)\| = 1, \quad \forall t \in [0,1].
    \]
   
    Then two pairs of arcs $(x_0 x_1, y_0 y_1)$, $(x_0 x_1, z_0 z_1)$ are complementary by construction. The pair $(y_1 x_1, x_1 z_1)$ is complementary, since the first arc transforms into the second arc by clockwise 1-rotation (Definition \ref{def_rotation}).
    
    Further, assume that the points $x_1, y_1, z_1, y_0=z_0$ are not bichromatic.  It follows from Proposition~\ref{prop_ind} that the indices of the arcs $y_0 y_1$, $z_0 z_1$  are either the same or differ by one, because they are both complementary to the arc $x_0 x_1$ and have a common origin $y_0=z_0$. Hence,
    \begin{equation}
    \operatorname{Ind} y_1 z_1 = \operatorname{Ind} y_1 y_0 +\operatorname{Ind} z_0 z_1 = - \operatorname{Ind} y_0 y_1 +\operatorname{Ind} z_0 z_1 \in \{-1,0,1\}. 
    \label{est1}
    \end{equation}
    On the other hand, $\varphi(y_1)\neq \varphi(x_1)$, $\varphi(x_1)\neq \varphi(z_1)$, and the arcs $(y_1 x_1, x_1 z_1)$ are complementary, hence
    \[
        |\operatorname{Ind} y_1 x_1| \geq 1, \quad | \operatorname{Ind} x_1 z_1| \geq 1, \quad | \operatorname{Ind} y_1 x_1- \operatorname{Ind} x_1 z_1| \leq 1. 
    \]

Non-zero integers $\operatorname{Ind} y_1 x_1$ and $\operatorname{Ind} x_1 z_1$ can differ by 1 only if they have the same sign. Now observe that
    \begin{equation}
        | \operatorname{Ind} y_1 z_1 | = | \operatorname{Ind} y_1 x_1 +\operatorname{Ind} x_1 z_1 | \geq 2.
\label{est2}
    \end{equation}
      We have a contradiction between \eqref{est1} and \eqref{est2}.

It remains to note that if there are bichromatic points among $x_1, y_1, z_1, y_0=z_0$, then it is sufficient to apply a $\alpha$-rotation, $0<\alpha<\varepsilon$ (see Definition \ref{def_rotation}) to the points considered on the circle $T_1(v)$. Here we move the points considered, but do not change the coloring of the circles. Since the set of bichromatic points on the circle is finite, $\alpha$ can be chosen in such a way that that the images $x'_1, y'_1, z'_1, y'_0=z'_0$ are not bichromatic. The complementarity conditions of the arcs in this case are preserved regardless of a small displacement, because
\[
    \|y'(t)-y(t)\|=\|z'(t)-z(t)\|=\alpha<\varepsilon, \; \forall t\in [0,1],
\]
and by the triangle inequality
\[
    1-\varepsilon < \|x(t)-y'(t)\|< 1+\varepsilon, \quad 1-\varepsilon < \|x(t)-z'(t)\|< 1+\varepsilon.\qedhere
\]
\end{proof}


\begin{remark}
    Computer experiments similar to those performed in \cite{warsaw,parts2023} show that in the  case of Euclidean distance $\chi_{[1-\varepsilon,1+\varepsilon]}(B(s))=4$ at $0.8\leq s \leq 2.65$ and $\varepsilon=0.003$. This can also be true when $\varepsilon \to 0$. The exact values of the upper and lower bounds of the parameter $s$ at which the chromatic number is 4 are not mentioned in this paper.
\end{remark}

\section{Proof of the main result}

\begin{lemma}
If a proper discrete coloring of the plane with 6 colors exists, then it contains such a pair of trichromatic points $u,v$  that \[C(u)=C(v), \quad 1<\|u-v\|<2.\]
\label{final}
\end{lemma}
\begin{figure}
    \centering
    \includegraphics[width=6cm]{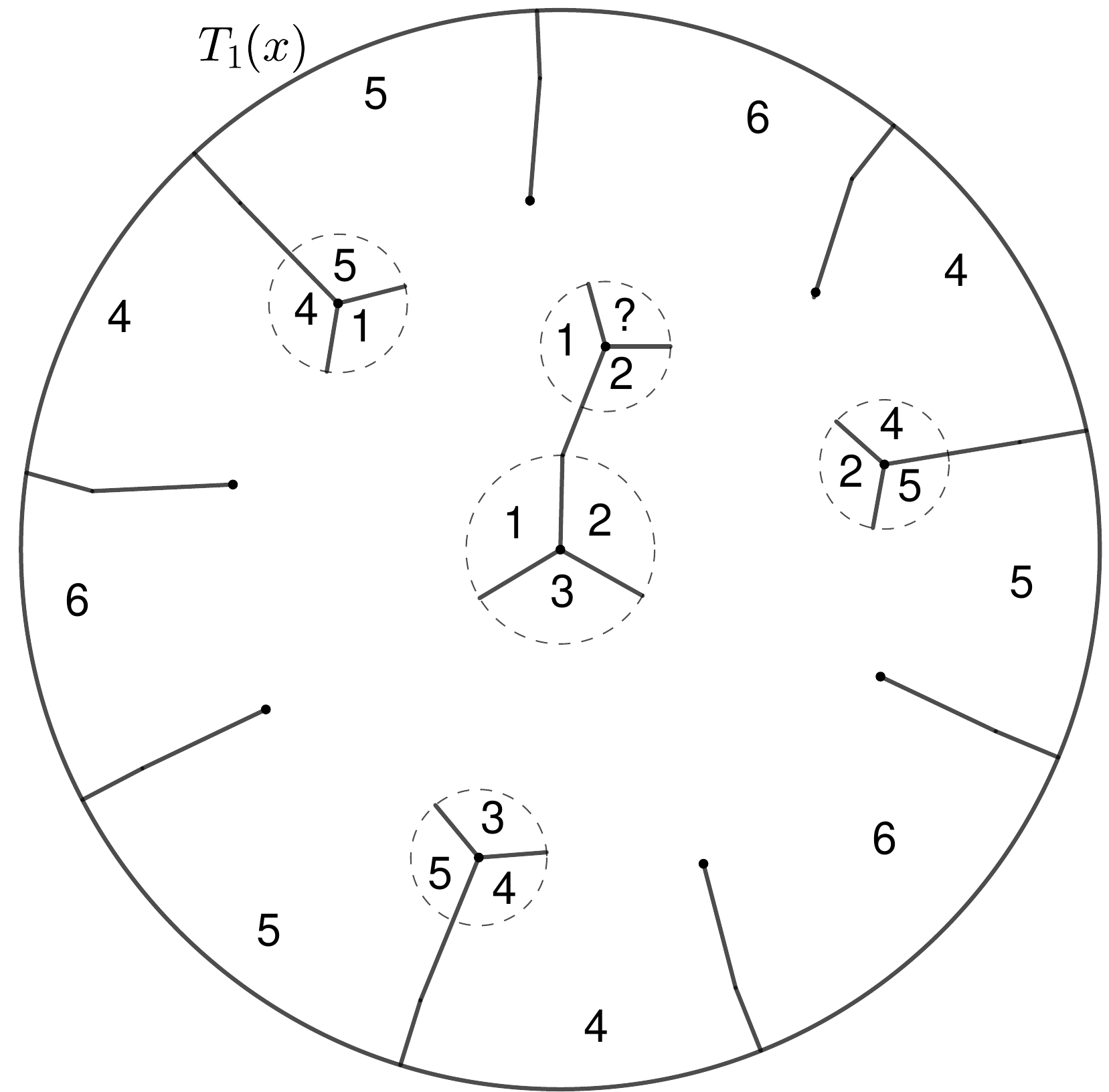}
    \caption{Search for 10+1 trichromatic points inside a circle}
    \label{fig_10pts}
\end{figure}
\begin{proof}
Consider a trichromatic point $x$. Let $C(x)=\{1,2,3\}$.  The circle $T_1(x)$ is colored in 4, 5, 6. According to Proposition~\ref{multicol}, the circle $T_1(x)$ contains at least 9 bichromatic points belonging to the boundaries between regions. According to Proposition~\ref{prop_connect}, for each of the multicolors $\{4,5\}$, $\{4,6\}$, $\{5,6\}$ there are three points with pairwise distances greater than 1, which are not connected by a boundary between regions with a point of the same multicolor on $T_1(x)$. So regions of color 4 adjacent to different points are at least $1+\varepsilon$ apart, and the same statement is true for colors 5 and 6.

    The neighborhood of $T_1 (x)$ contains three boundaries between areas of colors $4,5$ that start at bichromatic points constructed according to Proposition~\ref{multicol}. Let us extend them inside the unit disk centered at $x$. According to Proposition~\ref{prop3col1}, these boundaries cannot end at a point that has a multicolor $\{4,5,6\}$.  Consequently, among these three trichromatic points either some two are of the same multicolor, or their multicolors are  $\{4,5,1\}$, $\{4,5,2\}$, $\{4,5,3\}$. In the latter case, the distance between two trichromatic points $u, v$ with the same multicolor cannot be less than 1, because the bichromatic boundaries are continuous. On the other hand,  $\|u-v\| <2 $  because $u,v$ are inside $T_1(x)$. 

    By reasoning similarly, we get that if no pair of points with the same multicolor is found, then a disk of unit radius centered at $x$ contains at least 10 trichromatic points of pairwise different multicolors
    \[
     \{1,2,3\}, \{4,5,p\}, \{4,6,p\}, \{5,6,p\}, \quad p \in \{1,2,3\}.
    \]

    In addition, all regions are 1-connected, and at least one of the boundaries between regions of colors $1,2$, or $1,3$, or $2,3$  should end at the border of one of the colors $4,5,6$. (Fig.~\ref{fig_10pts}). Indeed, consider a maximal region $R$ colored in $1, 2, 3$ and containing $u$. Along the boundary of $R$ there cannot be tiles of only one of the colors $1, 2, 3$, because the monochromatic regions are 1-connected. Hence, the boundary between a pair of colors $1, 2, 3$ faces the boundary of the region at the trichromatic point $p$. The third color at $p$ must be one of $4,5,6$. But having an additional trichromatic point with a multicolor $\{1,2,q\}, \; q \in \{4,5,6\}$ in the circle contradicts Proposition~\ref{prop3col1}.
\end{proof}

Let us proceed to the proof of Theorem~\ref{thm_main}.

\begin{proof}
By Lemma~\ref{final} in any proper discrete coloring of the plane in 6 colors with parameter $h>0$, for arbitrarily small $h$ there will be trichromatic points $u, v$ with the same multicolor, with $1<\|u-v\|<2$. Consider a pair of circles $T_1(u)$, $T_1(v)$. Their union is a bicycle. It must be colored in three colors, which is impossible according to the Lemma~\ref{bicycle}. Then it follows from Proposition~\ref{prop_discrete} that an arbitrary proper coloring in 6 colors with forbidden interval $[1-\varepsilon,1+\varepsilon]$ also does not exist. Finally, it remains to note that the transition from a strictly convex norm to an arbitrary norm can be easily realized by Lemma~\ref{lemma_strictly_convex}.  
\end{proof}

Finally, we prove Theorem~\ref{thm_cover}.

\begin{proof}
    Denote by $U_r(x)$ closed unit disc of radius $r$ centered at $x$. 
    Let us recall the statements used in the proof of Theorem \ref{thm_main} and determine the radius of the disc $U_r(0)$, in which all the constructions lie. First, the trichromatic point $x$ can be found at a distance at most 1 from the origin by Lemma \ref{lemma_trichromatic_pt}. Second, the circle $T_1(x)$ lies in $U_2(0)$. Finally, a second circle $T(y)$ with the center in the interior of $T_1(x)$ lies in $U_3(0)$. Hence, all reasoning is valid for the disc of radius 3, which was mentioned in the formulation of Theorem~\ref{thm_cover}.

Consider the covering of $U_3(0)$ by six closed sets $A_1, \dots, A_6$. For some $\varepsilon>0$ consider a tesselation of the plane by hexagons of diameter $\varepsilon/3$. For a hexagon $H$ one can choose a color $c(H)=i$ such that $H$ has a non-empty intersection with $A_i$. By the formulation of Theorem \ref{thm_cover}, such a coloring is possible for all hexagons having a common point with the disc $U_3(0)$.
According to Theorem \ref{thm_main} and the estimate on the diameter of the circle above, there exists such a pair of points $x(\varepsilon), y(\varepsilon) \in U_3(0)$ that violate the conditions of proper coloring, i.e. $\|x(\varepsilon)-y(\varepsilon)\| \in [1-\varepsilon, 1+\varepsilon]$, and $c(x(\varepsilon))=c(y(\varepsilon))$.

Consider an infinite sequence of positive real numbers $\{\varepsilon_k\}_{k=1}^{\infty}$, $\lim_{k \to \infty} \varepsilon_k=0$.  For each $\varepsilon_k$ define an (ordered) pair of points $(x_k,y_k)$, $x_k = x(\varepsilon_k)$, $y_k = y(\varepsilon_k)$. Each pair can be considered as a point of the set
\[
    U^2 = U_3(0) \times U_3(0) \subset \mathbb{R}^4.
\]
Define the \emph{product metric} on $\mathbb{R}^4$, constructed via the $\max$-norm applied to a given 2-norm on subspaces (\cite{deza2009encyclopedia}, p.~91):
\[
    \rho\left((u_1,v_1), (u_2,v_2)\right) = \max\{\|u_1-u_2\|_U, \|v_1-v_2\|_U\}.
\]

Since $U^2$ is closed and bounded, there exists a limit point, i.e. an ordered pair $(x^*, y^*) \in U^2$. Further,  $1-\varepsilon_k<\|x^*-y^*\|<1+\varepsilon_k$, $\varepsilon_k \to 0$, and we have an equality $\|x^*-y^*\|=1$.

Observe that from the sequence of pairs $(x_k,y_k)$ converging to $(x^*,y^*)$, we can select a monochromatic subsequence (e.g. of colour $i$) having the same limit. But this means that the distance from each of the points $x^*$, $y^*$ to the closed set $A_i$ is less than $\varepsilon_j$, $j \in J$ for some infinite subset of indices $J \subset \{1,2,3, \dots\}$, and $x^* \in A_i$, $y^* \in A_i$, $\|x^*-y^*\|=1$, as required. 
\end{proof}

\section{Conclusion}


It seems that the arguments used in this paper can be applied to a wider class of problems. For example, Property \ref{property3} is not necessary. The different condition, namely, that $6$ arcs of unit diameter do not cover the unit circle, is sufficient. This is true for negative curvature of the surface and false for positive curvature. Nevertheless, let us leave the formulations of maximum generality outside the scope of this paper.

In the private communication, D{\"o}m{\"o}t{\"o}r P{\'a}lv{\"o}lgyi pointed out that the same formulation for the sphere  $S^2(r)$ of a sufficiently large radius $r$ can lead to a different answer. As was shown in work by  P{\'e}ter Ágoston~\cite{agoston2020lower}, at least $8$ colors are required if two the same colors are at a distance more than $1$ apart. Such coloring, following paper by C. Thomassen~\cite{thomassen1999nelson}, is called \emph{nice}. It is unclear what happens if we drop this condition. In the case of a sphere, is it true that the lower estimate for nice colorings equals the estimate for colorings with a forbidden interval?

In the formulation of the questions, we assume the metric to be Euclidean, although it makes sense to consider the general case as well.

\begin{question}
    Is it true that there is some value $r_0$ such that
    \[
    \chi_{[1-\varepsilon,1+\varepsilon]}(S^2(r)) \geq 8,
    \]
    when $r>r_0$?
\end{question}

Let us formulate an open question on a more complicated relaxation of the Hadwiger--Nelson problem, considered earlier in \cite{slice2,slice3}.

\begin{question}
Is it true that 
\[
    \chi_{\{1\}}(\mathbb{R}^2 \times [0,\varepsilon]^k) = 7.
\]
for small enough $\varepsilon>0$ and some $k$ independent of $\varepsilon$?
\end{question}

Note that in \cite{slice2,slice3} it was shown that
\[
    \chi_{\{1\}}(\mathbb{R}^2 \times [0,\varepsilon]^2) \geq 6, \quad   \chi_{\{1\}}(\mathbb{R}^3 \times [0,\varepsilon]^6) \geq 10.
\]

Speaking of three-dimensional Euclidean space, it is known that the coloring of the filling by distorted  truncated octahedrons is proper even if the interval is forbidden, i.e.
\[
     10 \leq \chi_{[1-\varepsilon,1+\varepsilon]}(\mathbb{R}^3) \leq 15,
\]
if $\varepsilon$ is sufficiently small \cite{Coulson,radoivcic2003note}. Apparently, in higher dimensions it is much more difficult to approach the exact estimate, but the following question seems natural.

\begin{question}
Is it true that
    $\chi_{[1-\varepsilon,1+\varepsilon]}(\mathbb{R}^3) = 15$
    for sufficiently small $\varepsilon$?
\end{question}

In the recent paper by N. Alon,  M. Buci{\'c} and  L. Sauermann  it was shown that for a ``typical'' Minkowski norm a finite graph of unit distances in $\mathbb{R}^2$ cannot have a chromatic number greater than four~\cite{alon2025unit}. Thus, all norms for which the chromatic number of the plane is greater than 4, are in some sense special, namely, in the sense of the first Baire category. For example, for any 5-chromatic unit distance graph in the Euclidean metric there exists an infinite family of norms in which it is also a unit distance graph, but this family is still a meagre set. Separate cases are 5-chromatic graphs constructed for a polygonal ``circle''~\cite{exoo2021chromatic}. There are also several upper estimates of the form $\chi(\mathbb{R}^2; F) \leq 6$ for norms of this type~\cite{geher2023note}.  The existence of specific norms in which the chromatic number of the plane would be exactly 5, 6 or 7 remains an open question. Unfortunately, it is far from always possible to construct a norm in which a 6-chromatic or 7-chromatic Euclidean $\varepsilon$-distance graph would be a ``honest'' unit distance graph. Apparently, this is not true for the constructive version of the proof in this paper.

\textbf{Acknowledgements}. 
The author is grateful to the anonymous reviewer whose comments allowed to correct numerous inaccuracies and fill gaps in the proofs. The author would like to thank Danila Cherkashin for helpful discussions and comments on the text, Jaan Parts for the experimental evidence that forced the writing of this paper and for pointing out some inaccuracies, D{\"o}m{\"o}t{\"o}r P{\'a}lv{\"o}lgyi for the formulation of a related question. Certainly, the programmers who created dynamic geometry software, in particular GeoGebra, should also be thanked for the idea of the proof in Section~\ref{sect_bike}.

\bibliographystyle{plain}
\bibliography{main}

\end{document}